\documentclass[11pt]{amsart}
\usepackage{amsfonts,amssymb,amsthm,eucal}

\newtheorem{thm}{Theorem}

\newtheorem{lemma}[thm]{Lemma}

\newcommand{\R}{\mathbb{R}}
\newcommand{\E}{\mathbb{E}}

\newcommand{\Prob}{\mathbb{P}}

\newcommand{\C}{\mathbb{C}}

\newcommand{\inprod}[2]{\left\langle #1, #2 \right\rangle}

\newcommand{\abs}[1]{\left\vert #1 \right\vert}
\newcommand{\norm}[1]{\left\Vert #1 \right\Vert}
\newcommand{\nnorm}[1]{\norm{#1}'}

\newcommand{\eps}{\varepsilon}

\allowdisplaybreaks[4]

\author{Elizabeth S.\ Meckes}
\thanks{E.\ Meckes's research is supported by an American Institute of Mathematics Five-year Fellowship and NSF grant DMS-0852898.}
\address{Case Western Reserve University, Cleveland, OH 44106, U.S.A.}
\email{elizabeth.meckes@case.edu}

\author{Mark W.\ Meckes}
\thanks{M.\ Meckes's research is supported by NSF grant DMS-0902203.}
\address{Case Western Reserve University, Cleveland, OH 44106, U.S.A.}
\email{mark.meckes@case.edu}

\title{Another observation about operator compressions}

\begin{document}

\begin{abstract}
Let $T$ be a self-adjoint operator on a finite dimensional Hilbert
space. It is shown that the distribution of the eigenvalues of a
compression of $T$ to a subspace of a given dimension is almost the
same for almost all subspaces. This is a coordinate-free analogue of a
recent result of Chatterjee and Ledoux on principal submatrices. The
proof is based on measure concentration and entropy techniques, and
the result improves on some aspects of the result of Chatterjee and
Ledoux.
\end{abstract}

\maketitle


\section{Introduction}\label{S:intro}

Let $T$ be an operator on a (real or complex) $n$-dimensional Hilbert
space $\mathcal{H}$, and let $E \subseteq \mathcal{H}$ be a subspace.
The \emph{compression} of $T$ to $E$ is the operator $T_E = \pi_E
T\vert_E = \pi_E T \pi_E^*$ on $E$, where $\pi_E:\mathcal{H}\to E$ is
the orthogonal projection. The spectral distribution of a self-adjoint
operator $T$ is the probability measure on $\R$
\[
\frac{1}{n} \sum_{i=1}^n \delta_{\lambda_i(T)},
\]
where $\lambda_1(T)\ge \dotsb \ge \lambda_n(T)$ are the eigenvalues
of $T$, counted with multiplicity.

The following result shows that for $1\le k \le n$ and a self-adjoint
operator $T$ on an $n$-dimensional Hilbert space $\mathcal{H}$, the
empirical spectral distribution of the compression $T_E$ is almost the
same for almost every $k$-dimensional subspace $E \subseteq
\mathcal{H}$.  The notations $\sigma_k$ and $\rho$ are explained after
the statement of the theorem; $d_1$ denotes the Kantorovich-Rubinstein
metric on probability measures, also defined below.

\begin{thm}\label{T:main}
Let $\mathcal{H}$ be an $n$-dimensional Hilbert space, $T$ a
self-adjoint operator on $\mathcal{H}$, and $1\le k \le n$. Let $E$ be
a $k$-dimensional subspace of $\mathcal{H}$ chosen at random with
respect to the rotationally invariant probability measure on the
Grassmann manifold.  Let $\mu_E$ be the empirical spectral
distribution of the compression of $T$ to $E$, and let $\mu=\E
\mu_E$. Then
\begin{equation}\label{E:mean}
\E d_1(\mu_E,\mu) \le c_1 \frac{\sigma_k(T)^{4/7} \rho(T)^{3/7}}{(kn)^{2/7}}
\end{equation}
and
\begin{equation}\label{E:tail}
\Prob \left[ d_1(\mu_E,\mu) \ge c_1 \frac{\sigma_k(T)^{4/7}
    \rho(T)^{3/7}}{(kn)^{2/7}} + t \right]
  \le c_2 \exp \left[-c_3 \frac{kn}{\sigma_k(T)^2} t^2\right]
\end{equation}
for every $t>0$, where $c_1,c_2,c_3>0$ are absolute (computable)
constants.
\end{thm}

Here $\rho(T)=\frac{1}{2}(\lambda_1-\lambda_n)$ denotes one half the
spectral diameter of $T$ (which is different in general from the
classical spectral radius); it is easy to check that $\rho(T)$ is the
distance of $T$ from the space of real scalar operators with respect
to the operator norm.  For $1\le k \le n$,
\[
\sigma_k(T) = \inf_{\lambda \in \R} \sqrt{\sum_{i=1}^k s_i(T-\lambda I)^2},
\]
where $s_1\ge \dotsb \ge s_n \ge 0$ denote singular values.  That is,
$\sigma_k(T)$ is the distance of $T$ from the space of real scalar
operators with respect to the norm $\norm{T}_{(k),2} =
\sqrt{\sum_{i=1}^k s_i(T)^2}$.

The space of probability measures (with finite first moment) on $\R$
is equipped with the Kantorovich-Rubinstein or $L^1$-Wasserstein
distance $d_1$, which may be equivalently defined in the following three
ways:
\begin{equation}\label{E:KR}
d_1(\mu,\nu) = \inf_{\pi} \int_{\R\times \R} \abs{x-y} \ d\pi(x,y)
= \sup_f \left(\int f \ d\mu - \int f \ d\nu\right) 
= \norm{F_\mu - F_\nu}_{L^1(\R)}.
\end{equation}
Here $\pi$ varies over all probability measures on $\R\times \R$ with
marginals $\mu$ and $\nu$; $f$ varies over all Lipschitz continuous
functions $\R\to \R$ with Lipschitz constant at most $1$; and $F_\mu$,
$F_\nu$ are the cumulative distribution functions of $\mu$, $\nu$. All
three characterizations will be used in this note; for the equalities
see \cite[Chapter 1]{Villani}.

Theorem \ref{T:main} is a coordinate-free analogue of a recent result
of Chatterjee and Ledoux \cite{CL}, which considered the empirical
spectral measure of a random $k\times k$ principal submatrix of a
fixed $n\times n$ Hermitian matrix. The approach taken in \cite{CL} is
rather different than the one taken here; the result of \cite{CL} is
also given in terms of the Kolmogorov distance between measures,
rather than Wasserstein distance.  See section \ref{S:discussion}
below for a more detailed comparison of the results.

\section{Proof of Theorem \ref{T:main}}\label{S:proof}

Throughout this section let $\mathcal{H}$ and $T$ be fixed, and let
$\mu_E$ and $\mu$ be as defined in the statement of the theorem.  For
brevity we write $\sigma_k=\sigma_k(T)$ and $\rho=\rho(T)$.  The
notation $\lesssim B$ means $\le c B$, where $c>0$ is some absolute
constant.

Recall that the Grassmann manifold $G_k(\mathcal{H})$ of
$k$-dimensional subspaces of $\mathcal{H}$ is equipped with the metric
\[
d(E,F) = \inf \sqrt{\sum_{i=1}^k \norm{e_i - f_i}^2},
\]
where the infimum is over all orthonormal bases $\{e_1,\dotsc, e_k\}$
and $\{f_1,\dotsc,f_k\}$ of $E$ and $F$ respectively.

\begin{lemma}\label{T:Lipschitz}
For any $E,F\in G_k(\mathcal{H})$, $d_1(\mu_E,\mu_F) \le
\frac{2\sigma_k}{\sqrt{k}} d(E,F)$.
\end{lemma}

\begin{proof}
Define a coupling $\pi$ of $\mu_E$ and $\mu_F$ by
$\pi = \frac{1}{k} \sum_{i=1}^k \delta_{(\lambda_i(T_E),\lambda_i(T_F))}$.
Then
\[
d_1(\mu_E,\mu_F) \le \frac{1}{k} \sum_{i=1}^k
\abs{\lambda_i(T_E)-\lambda_i(T_F)} \le \sqrt{\frac{1}{k} \sum_{i=1}^k
  \abs{\lambda_i(T_E)-\lambda_i(T_F)}^2}.
\]
Now if $\{e_1,\dotsc,e_k\}$ and $\{f_1,\dotsc,f_k\}$ are orthonormal
bases of $E$ and $F$, then the matrices of $T_E$ and $T_F$ with
respect to these bases are $\big[\inprod{T(e_j)}{e_i}\big]_{i,j=1}^k$,
$\big[\inprod{T(f_j)}{f_i}\big]_{i,j=1}^k$ respectively.  As a consequence 
of Lidskii's theorem (see \cite[$\S$ III.4]{Bhatia}), for any
$k\times k$ Hermitian matrices $A$ and $B$,
\[
\sqrt{\sum_{i=1}^k \abs{\lambda_i(A)-\lambda_i(B)}^2}
\le \norm{A-B}_{HS} = \sqrt{\sum_{i,j=1}^k \abs{a_{ij}-b_{ij}}^2}.
\]
Thus by the self-adjointness of $T$ and the Cauchy-Schwarz inequality,
\begin{equation*}\begin{split}
d_1(\mu_E,\mu_F)
&\le \sqrt{\frac{1}{k} \sum_{i,j=1}^k \abs{\inprod{T(e_j)}{e_i}
  - \inprod{T(f_j)}{f_i}}^2}
= \sqrt{\frac{1}{k} \sum_{i,j=1}^k \abs{\inprod{T(e_j)}{e_i-f_i}
  + \inprod{e_j-f_j}{T(f_i)}}^2} \\
& \le \sqrt{\frac{1}{k} \sum_{i,j=1}^k \left[\norm{T(e_j)}\norm{e_i-f_i}
  + \norm{T(f_i)}\norm{e_j-f_j}\right]^2} \\
& \le \frac{d(E,F)}{\sqrt{k}}\left(\sqrt{\sum_{j=1}^k \norm{T(e_j)}^2}
  + \sqrt{\sum_{i=1}^k \norm{T(f_i)}^2}\right)
\le \frac{2d(E,F)}{\sqrt{k}} \norm{T}_{(k),2}.
\end{split}\end{equation*}
Observing that $d_1(\mu_E,\mu_F)$ is invariant under addition of a real
scalar matrix to $T$, the lemma is proved.
\end{proof}

The same proof as above can be carried out (and is slightly simpler)
with the Kantorovich-Rubinstein distance replaced by the
$L^2$-Wasserstein distance, although this observation will not be used
here.

The following concentration inequality goes back to Gromov and Milman
\cite{GM}; see also section 2.1 of \cite{Ledoux} where it is pointed
out explicitly that the same result applies in the complex case.

\begin{thm}\label{T:GM}
Let $f:G_k(\mathcal{H}) \to \R$ be 1-Lipschitz with respect to the
metric $d$ on $G_k(\mathcal{H})$, and let $E\in G_k(\mathcal{H})$ be
distributed according to the rotation-invariant probability measure on
$G_k(\mathcal{H})$.  Then
\[
\Prob\big[\abs{f(E)-\E f(E)}\ge t\big] \lesssim \exp\big[-cnt^2\big]
\]
for $t>0$, where $c>0$ is an absolute constant.
\end{thm}

Observe that \eqref{E:mean}, Lemma \ref{T:Lipschitz}, and Theorem
\ref{T:GM} together imply \eqref{E:tail}, so it suffices now to prove
\eqref{E:mean}.

Let $E\in G_k(\mathcal{H})$ be distributed according to the
rotation-invariant probability measure on $G_k(\mathcal{H})$. For a
given function $f:\R\to \R$, define the random variable $X_f = \int f
d\mu_E - \int f d\mu$. By Lemma \ref{T:Lipschitz} and Theorem
\ref{T:GM}, for functions $f,g$
\begin{equation}\label{E:increment}
\Prob\big[\abs{X_f-X_g} \ge t\big] = \Prob\big[\abs{X_{f-g}}\ge t]
  \lesssim \exp\left[-c \frac{kn}{\sigma_k^2\abs{f-g}_L^2}\right]
\end{equation}
for $t \ge 0$, where $\abs{f}_L$ denotes the Lipschitz constant of
$f$. 

The inequality \eqref{E:increment} shows that the random process
$X_f$, indexed by some family $\mathcal{F}$ of Lipschitz continuous
test functions (to be determined), satisfies a subgaussian increment
condition with respect to the norm $\nnorm{\cdot} =
\frac{\sigma_k}{\sqrt{kn}}\norm{\cdot}_{C^1}$ on $\mathcal{F}$ (here, 
$\|f\|_{C^1}:=\max\{\|f\|_\infty,\|f'\|_\infty\}$, so that for $f\in C^1$,
$|f|_L\le \|f\|_{C^1}$). This
raises the possibility to estimate its expected supremum by Dudley's
entropy bound \cite{Dudley} (see also \cite{Talagrand}):
\begin{equation}\label{E:Dudley}
\E \sup_{f\in \mathcal{F}} X_f \lesssim \int_0^\infty
\sqrt{\log N(\mathcal{F}, \nnorm{\cdot}, \eps)} \ d\eps,
\end{equation}
where $N(\mathcal{F}, \nnorm{\cdot}, \eps)$ is minimum number of sets
of diameter $\eps$ with respect to $\nnorm{\cdot}$ needed to cover
$\mathcal{F}$. Since $\mu_E$ and $\mu$ are supported on
$[\lambda_n,\lambda_1]$,
\[
d_1(\mu_E,\mu) = \sup \big\{ X_f : \abs{f}_L\le 1\big\} =
\sup\big\{ X_f : \abs{f}_L\le 1, \norm{f}_\infty \le 2\rho\big\}.
\]
Thus to prove \eqref{E:mean}, it suffices to estimate $\E \sup_{f\in
  \mathcal{F}} X_f$ for $\mathcal{F} = \{ f : \norm{f}_{C^1} \le 1 +
2\rho\}$.  However, since $C^1$ is an infinite dimensional function
space, for this choice of $\mathcal{F}$ the covering numbers
$N(\mathcal{F}, \nnorm{\cdot}, \eps)$ in \eqref{E:Dudley} will always
be infinite for small $\eps$.

Instead, define $\mathcal{F} = \{ f : \norm{f}_{C^2} \le 1\}$, where
$\|f\|_{C^2}:=\max\{\|f\|_\infty,\|f'\|_\infty,\|f''\|_\infty\}$. The
covering numbers $N(\mathcal{F}, \norm{\cdot}_{C^1}, \eps)$ can be
estimated using the methods of \cite[$\S$ 2.7]{VW}; see \cite{EMeckes}
for explicit estimates which, combined with
\eqref{E:Dudley} and a linear change of variables, yield
\begin{equation}\label{E:C2}
\E \sup \big\{X_f : \norm{f}_{C^2} \le 1 \big\}
  \lesssim \frac{\sigma_k}{\sqrt{kn}}\int_0^{1} \sqrt{1 + \log\frac{1}{\eps} 
    + \frac{1}{\eps}(\rho + 1)} \ d\eps
  \lesssim \frac{\sigma_k \sqrt{\rho + 1}}{\sqrt{kn}}.
\end{equation}

The bound \eqref{E:mean} is now derived from \eqref{E:C2} via a
smoothing and scaling argument.  Fix $f:\R\to \R$ with $\abs{f}_L \le
1$ and $\norm{f}_\infty \le 2\rho$. Let $\varphi:\R\to \R$ be a smooth
probability density with finite first absolute moment and $\varphi'\in
L^1(\R)$. For $t>0$ define $\varphi_t(x) = \frac{1}{t}
\varphi(\frac{x}{t})$, and let $g_t = f * \varphi_t$. Then
\begin{equation}\label{E:g_t}
\norm{g_t}_\infty \le \norm{f}_\infty \norm{\varphi_t}_1 \le 2\rho,
\quad\qquad \norm{g_t'}_\infty \le \abs{f}_L \norm{\varphi_t}_1 \le 1, \qquad
\quad
\norm{g_t''}_\infty \le \abs{f}_L \norm{\varphi_t'}_1 \lesssim
\frac{1}{t}.
\end{equation}
Now for any probability measure $\nu$ on $\R$,
\[
\abs{\int f \ d\nu - \int g_t \ d\nu} = \abs{\int \int [f(x) -
    f(x-y)]\varphi(y) \ dy \ d\nu(x)} \lesssim t.
\]
Thus
\begin{equation*}\begin{split}
\abs{X_f} & \le \abs{\int f\ d\mu_E - \int g_t\ d\mu_E}
  + \abs{\int g_t \ d\mu_E - \int g_t \ d\mu} 
  + \abs{\int g_t \ d\mu - \int f\ d\mu}\\
& \lesssim t + \norm{g_t}_{C^2} 
  \sup \big\{X_h : \norm{h}_{C^2} \le 1 \big\},
\end{split}\end{equation*}
and so by \eqref{E:C2} and \eqref{E:g_t},
\[
\E d_1(\mu_E,\mu) \lesssim t + \left(1+2\rho+ \frac{1}{t}\right) 
  \frac{\sigma_k\sqrt{\rho+1}}{\sqrt{kn}}.
\]
Picking $t$ of the order
$\frac{\sqrt{\sigma_k}(\rho+1)^{1/4}}{(kn)^{1/4}}$ yields
\begin{equation}\label{E:needs-rescaling}
\E d_1(\mu_E,\mu) \lesssim \frac{\sqrt{\sigma_k}(\rho+1)^{1/4}}{(kn)^{1/4}}
  + \frac{\sigma_k (\rho+1)^{3/2}}{\sqrt{kn}}.
\end{equation}

Now apply \eqref{E:needs-rescaling} with the operator $T$ replaced by
$sT$ for $s>0$. It is easy to check that the Kantorovich-Rubinstein
distance $d_1(\mu_E,\mu)$ is homogeneous with respect to this rescaling,
as are $\sigma_k$ and $\rho$. Thus one obtains
\[
\E d_1(\mu_E,\mu) \lesssim \frac{1}{s} 
  \left(\frac{\sqrt{s\sigma_k}(s\rho+1)^{1/4}}{(kn)^{1/4}}
  + \frac{s\sigma_k (s\rho+1)^{3/2}}{\sqrt{kn}}\right).
\]
Picking $s$ of the order $\frac{(kn)^{1/7}}{\sigma_k^{2/7} \rho^{5/7}}$ now
yields \eqref{E:mean}. \qed

\section{Discussion}\label{S:discussion}

In \cite{CL}, Chatterjee and Ledoux proved a version of Theorem
\ref{T:main} for principal submatrices.  Namely, let $\mathcal{H} =
\C^n$, and suppose $E$ is now uniformly distributed among
$k$-dimensional \emph{coordinate} subspaces of $\C^n$. Then
\cite{CL} shows that
\begin{equation}\label{E:CL-tail}
\Prob\big[d_\infty(\mu_E,\mu) \ge k^{-1/2} + t\big] \le 12\sqrt{k}
e^{-t\sqrt{k/8}}
\end{equation}
for $t>0$, and consequently
\begin{equation}\label{E:CL-mean}
\E d_\infty(\mu_E,\mu) \le \frac{13+\sqrt{8}\log k}{\sqrt{k}}.
\end{equation}
Here $d_\infty(\mu,\nu) = \norm{F_\mu - F_\nu}_\infty$ is the
Kolmogorov distance between probability measures $\mu$ and $\nu$ on
$\R$.

It is likely that the methods of this paper could be used to prove a
result in the setting of \cite{CL}, by replacing Theorem \ref{T:GM},
which follows from concentration inequalities on the unitary or
special orthogonal group, with an appropriate concentration inequality
on the symmetric group $S_n$.  Furthermore, it may be possible to
prove a result in the setting of this paper using methods related to
those of \cite{CL}, such as adapting the approach of
Chatterjee in \cite{Chatterjee}.  Below some quantitative 
comparison will be offered between
Theorem \ref{T:main} and the result of \cite{CL}, ignoring the fact
that the random subspace $E$ has a different distribution in each
setting.  In particular, the distribution of $E$ is probably
responsible for the difference between the subexponential tail decay
in \eqref{E:CL-tail} and the subgaussian tail decay in \eqref{E:tail}.

Before discussing more specific quantitative comparisons, we note that
the clearest difference between the two results is that ours is
coordinate-free.  While there are settings in which coordinates have 
meaning and thus coordinate-oriented results are natural, there are 
many settings in which there is no clearly preferred basis in which to
view an operator.  Take, for example, the Laplacian $\Delta$ on the sphere $
\mathbb{S}^{n-1}$. It has eigenvalues (up to sign convention) 
$0>-\lambda_1>-\lambda_2>\cdots\to-\infty$, and the corresponding 
eigenspaces are multidimensional.  If one took $\mathcal{H}$ to be the span of
the first $m$ eigenspaces, with $T=\Delta|_{\mathcal{H}},$
there is no canonical choice of basis within each eigenspace, and so 
it would seem more natural to consider compressions of $T$ to {\em all}
subspaces of a given dimension, rather than only to the coordinate 
subspaces for some choice of basis.

Comparisons of the results are made somewhat difficult as the 
Kantorovich-Rubinstein distance $d_1$ and the Kolmogorov distance
$d_\infty$ are not comparable in general. However, since the measures
here are all supported in the interval $[\lambda_n(T),\lambda_1(T)]$,
from the third representation of $d_1$ in \eqref{E:KR} one obtains the
estimate 
\begin{equation}\label{d1dinfty}
d_1(\mu_E,\mu) \le 2\rho(T) d_\infty(\mu_E,\mu)
\end{equation}
in the
present context. This estimate is related to a qualitative difference
between $d_1$ and $d_\infty$: whereas $d_1$ is homogeneous with
respect to a rescaling of the supports of measures (a fact which was
exploited in the proof of Theorem \ref{T:main}), $d_\infty$ is
invariant under rescaling. Which behavior is more convenient may vary
by the context.

Inequality \eqref{d1dinfty} makes 
some quantitative comparisons between the results of \cite{CL} and
Theorem \ref{T:main} possible.  Observe that \eqref{E:CL-tail} and
\eqref{E:CL-mean} only yield nontrivial information if $k \gg 1$
(which of course requires $n\gg 1$), whereas under appropriate
scaling, Theorem \ref{T:main} is nontrivial for $n\gg 1$ even if $k$
is small.  In particular, \eqref{E:CL-tail} and \eqref{d1dinfty} imply
that the fluctuations of $d_1(\mu_E,\mu)$ above its mean are of order
(ignoring logarithmic factors) at most $k^{-1/2}\rho(T)$, whereas
\eqref{E:tail} together with the general estimate $\sigma_k(T)\le
\sqrt{k}\rho(T)$ yields fluctuations of order at most $(kn)^{-1/2}
\sigma_k(T) \le n^{-1/2}\rho(T)$.  

The issue of the expected distance is more complicated.
The general
estimate $\rho(T) \le \sigma_k(T)$ and inequalities \eqref{E:CL-mean}
and \eqref{d1dinfty} imply that
\begin{equation}\label{E:CL-d1}
\E d_1(\mu_E,\mu) \le c \frac{\rho(T) \log k}{\sqrt{k}}
\le c \frac{\sigma_k(T)^{4/7}\rho(T)^{3/7}\log k}{\sqrt{k}},
\end{equation}  which is slightly weaker than \eqref{E:mean} for $k$
large (in which case the lossy estimates used to arrive at \eqref{E:CL-d1}
mean that the comparison should probably not be taken too seriously) 
and significantly weaker for $k$ small.  
Since the
different distributions of $E$ are being ignored here there is little
point in making the comparison very precise.

Finally, the comparison of fluctuations highlights that the methods
of this paper are more sensitive to the proximity of $T$ to the space
of scalar operators.  If $T$ is a (real) scalar operator then $\mu_E$
is a constant point mass, so it is natural to expect that if $T$ is
nearly scalar in some sense then $\mu_E$ will be more tightly
concentrated then in general.  The results of \cite{CL} do not
directly reflect this at all, although the estimate \eqref{d1dinfty}
allows one to insert this effect by
hand when changing metrics.  However, $\sigma_k(T)$ provides a sharper
measure than $\rho(T)$ of how close $T$ is to scalar, and in some cases
the bound $(kn)^{-1/2} \sigma_k(T)$ on the order of the fluctuations
may be even much smaller than $n^{-1/2}\rho(T)$. This is
the case, for example, if $T$ has a large number of tightly clustered
eigenvalues with a small number of outliers.

\bibliographystyle{plain} 
\bibliography{compressions}

\end{document}